\documentclass[leqno, letterpaper, 11pt]{article}
\usepackage{graphicx} % Required for inserting images
\usepackage{amssymb}
\usepackage{amsmath}
\usepackage{amsthm}
\usepackage{tikz}
\usepackage{fullpage}
\usepackage{mathrsfs}
\usepackage{subfigure}
\usepackage{verbatim}
\usepackage{enumitem}
\usepackage{bbm}
\usepackage{tikz}
\usetikzlibrary{through,calc,positioning,decorations.pathreplacing, arrows.meta}
\usepackage{authblk}

\newcommand{\bZ}{{\mathbb Z}}

\newcommand{\cC}{{\mathcal C}}

\newcommand{\cU}{\mathcal U}

\numberwithin{equation}{section}
     
\newtheorem{theorem}{Theorem}[section]

\newtheorem{lemma}[theorem]{Lemma}

\newcommand{\prob}[1]{\mathbb{P}\left(#1\right)}
\definecolor{ashgrey}{rgb}{0.7, 0.75, 0.71}
\definecolor{faint}{rgb}{0.9, 0.9, 0.9}
\definecolor{lgrey}{rgb}{0.85, 0.85, 0.85}
\definecolor{dgrey}{rgb}{0.5, 0.5, 0.5}
\definecolor{mgrey}{rgb}{0.78, 0.78, 0.78}
\definecolor{lred}{rgb}{1, 0.8, 0.8}
\definecolor{dred}{rgb}{0.7, 0, 0}

\newcommand{\new}[1]{{\color{black}#1}}

\title{Polluted Modified Bootstrap Percolation}

\author[1]{Janko Gravner}
\author[2]{Alexander Holroyd}
\author[3]{Sangchul Lee}
\author[4]{David Sivakoff}
\affil[1]{Department of Mathematics, University of California, Davis, 
{\tt gravner@math.ucdavis.edu}}
\affil[2]{Department of Mathematics, University of Bristol, {\tt a.e.holroyd@bristol.ac.uk}}
\affil[3]{Computational Science Research Center, Korea Institute of Science and Technology,
{\tt sangchul87.lee@gmail.com}}
\affil[4]{Departments of Statistics and Mathematics, The Ohio State University, {\tt dsivakoff@stat.osu.edu}}

\date{\today}

\begin{document}

\maketitle
\begin{abstract}
In the polluted modified 
bootstrap percolation model, sites in the square lattice are independently initially occupied with probability $p$ or closed with
probability $q$.  A site
becomes occupied at a subsequent step if it is not closed and has at least one occupied nearest
neighbor in each of the two coordinates. We study the final density of occupied sites when $p$ and $q$ are both small. We show that 
this density approaches $0$ if $q\ge Cp^2/\log p^{-1}$ and $1$ 
if $q\le p^2/(\log p^{-1})^{1+o(1)}$. Thus we establish a logarithmic  correction in the critical scaling, which is known not to be present in
the standard model\new{, settling a conjecture of Gravner and McDonald from 1997}. 
\end{abstract}
\section{Introduction}

Bootstrap percolation has been extensively studied
as a model for nucleation and metastability. In the version 
considered in this paper, each site in the lattice $\bZ^2$
is in one of the three states: occupied (also called infected), 
closed (immune), or open (susceptible). 
Initially, the state of each vertex is chosen independently 
at random: occupied with probability $p$ and closed with 
probability $q$\new{, where $p+q\le 1$}.
We will assume throughout that the probability parameters $p$ and $q$ are small. After the initial configuration is chosen, it evolves according to a simple deterministic rule in which occupied and closed sites never change their states, but an open site may become occupied. In the {\it standard bootstrap percolation\/}, an open site at time $t$ becomes occupied at time $t + 1$ if and only if at least 2 of its 4 nearest neighbors are occupied at time $t$. In the 
{\it modified\/} version, this transition is further 
restricted: at each step, every open site that has at least one occupied nearest neighbor to the east or west and at least one occupied nearest neighbor to the north or south becomes occupied.

The standard \new{``unpolluted''} case with $q=0$, a variant of which was introduced in \cite{CLR}, has played a central role in the study of spatial dynamics in random settings. The first rigorous result, proved in \cite{vE}, established that, as soon as $p>0$, every site of $\bZ^2$
is eventually occupied with probability $1$. To probe the metastability 
properties of the model, a natural setting is to restrict the lattice 
to a finite square of large side length $n$. This lead to
increased understanding of the mechanism for nucleation and the 
exponential scaling function is now known  to a remarkable precision 
\new{[AL, Hol, GHM, HM, HT]}. Many variations were also extensively studied, e.g.~\cite{BBDM, BDMS, BBMS1, BBMS2}. A brief survey cannot do justice to the variety and 
depth of the results in this area, 
so we refer to the review \cite{Mor}
for further background. 

The ``polluted'' standard case with $q>0$ was introduced in 
\cite{GM}. Here, the main focus is the final density of occupied sites. The main result of \cite{GM} is that there exist constants $c,C>0$ so that when $(p,q)\to (0,0)$, 
\begin{equation}\label{eq:GM}
\prob{\text{the origin is occupied}}\to
\begin{cases}
1 & \text{if } q<cp^2\\
0 & \text{if } q>Cp^2\\
\end{cases}
\end{equation}
Rigorous extensions of this result have been relatively scarce, but  
higher-dimensional cases \new{\cite{GraH, GHS, DYZ}} and general two-dimensional models \cite{Gho} have been addressed recently.

 \begin{figure}[ht!]
     \centering
     \includegraphics[width=0.4\linewidth]{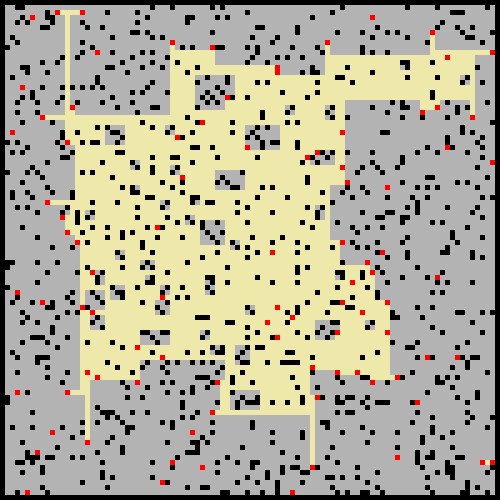}
     \caption{The dynamics with $p=0.1$, $q=0.01$ attempts to fill in a square.  Black sites are initially occupied, including all sites on the boundary, grey sites are eventually occupied, and red sites are closed.}
     \label{fig:sim}
 \end{figure}

For the modified model, the argument  
in \cite{GM}, which proves (\ref{eq:GM}) for the standard rule, allows the possibility 
of a logarithmic correction in the scaling (see also open 
problem (2) in \cite{GHS}). In this paper, we 
confirm that, indeed, $p^2$ in (\ref{eq:GM}) needs to be replaced 
by $p^2/\log p^{-1}$, although we have to allow a possible correction power of 
$\log\log p^{-1}$ in the lower bound. It is unclear whether 
such a power can be removed, and we leave this as a natural open question.
While the modified model is in many senses
similar to the standard one, notably in the dominant exponential term in the nucleation scaling \cite{Hol1}, 
there are important differences. The recent result on the scaling correction \cite{Har} is one example, and the present 
paper provides another. The cause is illustrated in Figure~\ref{fig:sim}, 
where narrow ``chimneys'' --- one site wide segments of open sites, with closed sites at their ends --- block the inward progress of occupation. 
\new{Unsurprisingly, the
chimney phenomenon is closely reflected in our proofs.}
%\new{Unsurprisingly, the chimney phenomenon plays an important role in our proofs.}
In the standard rule, occupation can easily invade such chimneys. We now formally state our results.

% Start with the initial configuration in $\mathbb{Z}^2$ has the product distribution with density $q$ of closed sites, $p$ of occupied (infected) sites and $1-p-q$ of open and not occupied sites. The modified bootstrap percolation dynamics proceeds on the open sites: at each step, every open site that has at least one occupied nearest neighbor to the east or west and at least one occupied nearest neighbor to the north or south becomes occupied.
\begin{theorem}\label{thm:ub}
    There exists a constant $C>0$ so that if $q\ge Cp^2 / \log (1/p)$, then
    \new{
    \begin{equation}\label{eq:ub}
        \lim_{q\to 0} \mathbb{P}_{p,q}\bigl((0,0) \text{ is eventually occupied}\,\bigr) = 0.
    \end{equation}
    }
\end{theorem}

\begin{theorem}\label{thm:lb}
    If 
    $$
    q\le \frac{p^2}{\log(1/p)(\log\log (1/p))^4},
    $$ 
    then
   \new{\begin{equation}\label{eq:lb}
       \lim_{p\to 0} \mathbb{P}_{p,q}\bigl((0,0) \text{ is eventually occupied}\,\bigr)= 1.
    \end{equation}
    }
\end{theorem}

\new{We prove these two theorems in the following two sections respectively.}

\section{Stopping the infection}\label{sec:ub}

%\new{Whenever we introduce one of our numerous rectangles, we use the following convention for its dimensions, for readability: if a dimension is given by a real number, it is to be interpreted 
%as the integer part of that number.}

Fix an odd integer $m\ge 1$ and \new{an integer $k\ge 2$}. 
Fix also $\epsilon, \delta>0$. 
Let \new{$M=\lfloor\delta p^{-1}\log p^{-1}\rfloor$} and 
    \new{$N=2m\lceil\epsilon/ (mp)\rceil$}. 
We partition 
the lattice into rectangles of width $M$
and height $N$, which we call {\em blocks}: for $z=(z_1,z_2)\in \bZ^2$, 
we let $R_z=(Mz_1,Nz_2)+([0,M-1]\times[0,N-1])$.

% Partition $\bZ^2$ into rectangles of width $\delta p^{-1}\log p^{-1}$
% and height $2\epsilon p^{-1}$, which we call {\em blocks}. We denote by $S_1$ the strip of width 
% $m$ and height $2k\epsilon/p$, centered at the origin, and by $S_2$ 
% the strip of height $m$ and width $2k\cdot \delta p^{-1}\log p^{-1}$, also centered at the origin. 
%For a block $R$, we denote $R_a$,  (resp.~$R_b$, $R_\ell$) to be the half %plane of all sites above (resp.~below, to the left of) $R$. 

A block $R_z$ is called \new{{\it safe\/}} if there exists an $x$ in the upper half of $R_z$ so that:\new{
\begin{itemize}
\item[(S1)] $x$ is closed; 
\item[(S2)] there is a rectangle of width $m$ and height  $k\lceil\epsilon p^{-1}\rceil$, whose top edge is included in the top edge of $R_z$, contains $x$ within its middle vertical line and contains no initially occupied sites; and 
\item[(S3)] there is a rectangle of width $k\cdot \lfloor\delta p^{-1}\log p^{-1}\rfloor$ and height $m$, whose left edge is contained within the bottom half of the left edge of $R_z$, and contains no initially occupied sites.
\end{itemize}}
%See Figure~\ref{fig:protection}.

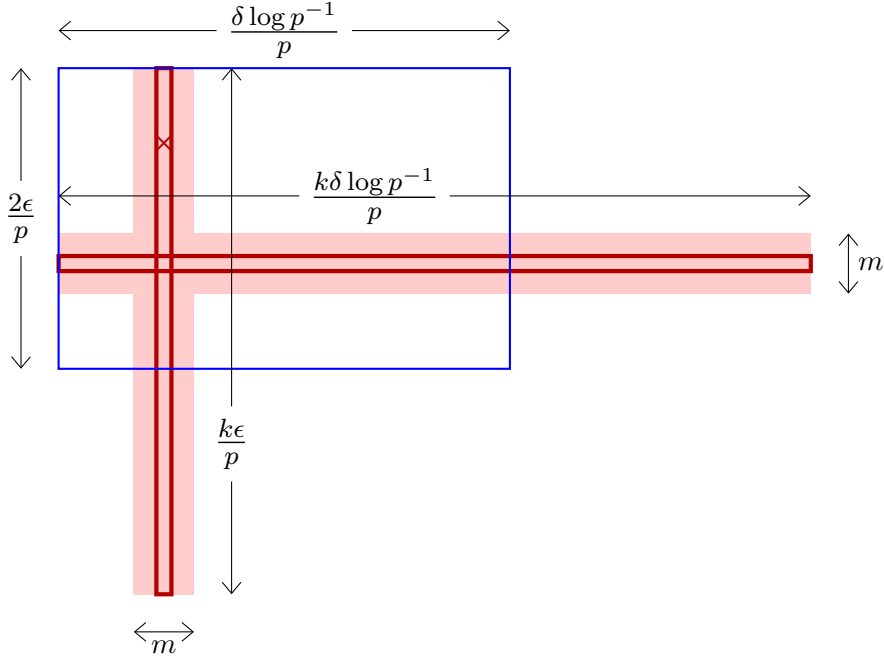
\begin{figure}%[ht!]
\begin{center}
\scalebox{1}{
\begin{tikzpicture}
\newcommand{\st}{Straight Barb[length=1.5mm, width=2.5mm]}

\fill[lred] (0,1) rectangle (10,1.8);
\fill[lred] (1,-3) rectangle (1.8,4);
\node[color=dred] at (1.4,3) {\large $\boldsymbol\times$};
\draw[color=dred, ultra thick](1+0.3,-3) rectangle (1.8-0.3,4);
\draw[color=dred, ultra thick](0,1.4-0.1) rectangle (10,1.4+0.1);
\draw[color=blue, thick](0,0) rectangle (6,4);
\draw [ \st-\st] (1,-3.5) -- (1.8,-3.5);
\node at (1.4,-3.7) {$m$};
\draw [ \st-\st] (10.5,1) -- (10.5,1.8);
\node at (10.8,1.4) {$m$};
\draw [ \st-] (0,2.3) -- (4.2-1,2.3);
\draw [ -\st] (4.2+1,2.3) -- (10,2.3);
\node at (4.2,2.3) {\scalebox{1.25}{$\frac{k\delta \log p^{-1}}{p}$}};
\draw [ \st-] (-0.5,0) -- (-0.5, 2-0.5);
\draw [ -\st] (-0.5,2+0.5) -- (-0.5, 4);
\node at (-0.5,2) {\scalebox{1.25}{$\frac{2\epsilon}{p}$}};
\draw [ \st-] (0,4.5) -- (3-0.9,4.5);
\draw [ -\st] (3+0.9,4.5) -- (6,4.5);
\node at (3,4.5) {\scalebox{1.25}{$\frac{\delta \log p^{-1}}{p}$}};
\draw [ \st-] (2.3,-3) -- (2.3, -1-0.5);
\draw [ -\st] (2.3,-1+0.5) -- (2.3, 4);
\node at (2.3,-1) {\scalebox{1.25}{$\frac{k\epsilon}{p}$}};
 
\end{tikzpicture}
}
\end{center}

\vspace{-0.5cm}

\caption{A \new{safe} block (outlined in blue) has a closed site marked by $\times$ and no occupied sites in the red protective region. Its core is outlined
by thicker red lines. \new{Dimensions of rectangles are given 
with integer parts omitted.}
} 
\label{fig:protection}
\end{figure}

If $R_z$ is \new{safe}, we call the union of the two rectangles from conditions \new{(S2)} and \new{(S3)} its {\em protective region}; the horizontal and vertical line segments  in the middle of the two rectangles together form its 
{\em core}; see Figure~\ref{fig:protection}.

\begin{lemma}\label{lemma:protection}
Fix $m$ and $k$, and any $\eta>0$. Then $\epsilon$, $\delta$ and a constant $C$ can be chosen so that, for any fixed $z$, 
$$\prob{R_z\text{ is \new{safe}}}\ge 1-\eta,$$
provided that 
$$
q\ge C\frac{p^2}{\log p^{-1}}
$$
and $p$ is small enough. 
\end{lemma}

\begin{proof} 
Note first that the probability that a fixed rectangle of dimensions 
in condition \new{(S2)} has an occupied site is at most
$$
\new{p\cdot m\cdot k\cdot \frac{2\epsilon}p=2mk\epsilon.}
$$
Therefore, the conditional probability that \new{(S2)} happens given that \new{(S1)} happens 
is at least $1-\eta/3$ for small enough $\epsilon$.
Further, for small $p$,
\begin{equation*}
\begin{aligned}
 \prob{\text{\new{(S3)} does not happen}}  &\le \left[1-(1-p)^{mk\delta p^{-1}\log p^{-1}}\right]^{\epsilon/(mp)}\\
 &\le 
 \left[1-\exp(-2mk\delta \log p^{-1})\right]^{\epsilon/(mp)}\\
 &=
 \left[1-p^{2mk\delta}\right]^{\epsilon/(mp)}\\
 &\le \exp\left[\frac \epsilon m p^{2mk\delta-1}\right],\\
 \end{aligned}
\end{equation*}
which goes to $0$ as $p\to 0$ provided that $2mk\delta<1$. Under this 
condition, \new{(S3)} therefore happens with probability at least $1-\eta/3$. 
Finally, 
$$
\frac 12|R_z|q\ge \new{\frac12}C\epsilon\delta, 
$$
and so \new{(S1)} happens with probability at least $1-\eta/3$ if $C$ is large 
enough. The claim follows.
\end{proof}

\begin{figure}%[ht!]
\begin{center}
\scalebox{0.65}{
\begin{tikzpicture}

\def\w{3}
\def\h{2}

\coordinate (x11) at (0,0);
\coordinate (x21) at (\w,0);
\coordinate (x22) at (\w,\h);
\coordinate (x31) at (2*\w,\h);
\coordinate (x41) at (3*\w,\h);
\coordinate (x42) at (3*\w,2*\h);
\coordinate (x43) at (3*\w,3*\h);
\coordinate (x51) at (4*\w,3*\h);
\coordinate (x61) at (5*\w,3*\h);

\path (x21) +(0.3*\w,0.7*\h) coordinate (c1);
\path (x21) +(0,0.3*\h) coordinate (c1p);
\coordinate (c1r) at (\w+0.3*\w,\h);
\path (x41) +(0.7*\w,0.8*\h) coordinate (c2);
\path (x41) +(0,0.4*\h) coordinate (c2p);
\coordinate (c2r) at (3*\w+0.7*\w,2*\h);
\path (x61) +(0.8*\w,0.6*\h) coordinate (c3);
\path (x61) +(0,0.2*\h) coordinate (c3p);
\coordinate (c3r) at (5*\w+0.8*\w,4*\h);

\draw[color=dred, line width=4]
  (\w+0.3*\w-1.3*\w,0.3*\h-\h)--(\w+0.3*\w,0.3*\h-\h);
  \draw[color=dred, line width=0.5]
  (\w+0.3*\w-1.3*\w,0.3*\h-\h)--(2*\w,0.3*\h-\h);

\draw[color=dred, line width=4]
 %(\w+0.3*\w,0.7*\h)--(\w+0.3*\w,\h-1.7*\h);
 (\w+0.3*\w,0.7*\h)--(\w+0.3*\w,\h-3*\h);
\draw[color=dred, line width=4]
  (\w+0.3*\w,0.3*\h)--(3*\w+0.7*\w,0.3*\h);
\draw[color=dred, line width=4]
  (3*\w+0.7*\w,\h+0.8*\h)--(3*\w+0.7*\w,0.3*\h-1.3*\h);
  %(3*\w+0.7*\w,\h+0.8*\h)--(3*\w+0.7*\w,0.3*\h);
\draw[color=dred, line width=4]
  (3*\w+0.7*\w,\h+0.4*\h)--(5*\w+0.8*\w,\h+0.4*\h);
  \draw[color=dred, line width=4]
  (5*\w+0.8*\w,\h+0.4*\h-0.4*\h)--(5*\w+0.8*\w,3*\h+0.6*\h);
  % (5*\w+0.8*\w,\h+0.4*\h)--(5*\w+0.8*\w,3*\h+0.6*\h);
  \draw[color=dred, line width=4]
  (5*\w+0.8*\w,3*\h+0.2*\h)--(5*\w+0.8*\w+\w,3*\h+0.2*\h);
  \draw[color=dred, line width=4]
  (5*\w+0.8*\w+\w,2*\h)--(5*\w+0.8*\w+\w,3*\h+\h);
  %(5*\w+0.8*\w+\w,3*\h+0.2*\h)--(5*\w+0.8*\w+\w,3*\h+\h);
  \draw[color=dred, line width=0.5]
  (5*\w+0.8*\w+\w,2*\h)--(5*\w+0.8*\w+\w,3*\h+\h);

%\coordinate (c1) at (x21)+(1.5,0.5);

\foreach \x in {x11, x21, x22,x31, x41,x42,x43,x51,x61}{
\draw[color=black, line width=0.5]
    (\x) + (0,0) rectangle +(\w,\h); 
}

%\draw (x21)+
%\node[color=dred] at +(2.5,0.5) {$\boldsymbol\times$};
\foreach \x in {c1,c2,c3}{
\node[color=dred] at (\x) {\huge$\boldsymbol\times$};
%\draw[color=dred, line width=0.5]
%    (\x) + (0,0) -- +(3*\w,0); 
}

\draw[color=dred, line width=0.5]
  (c1p)+(0,0)-- +(3*\w,0);
\draw[color=dred, line width=0.5]
  (c1r)+(0,0)-- +(0,-3*\h);
\draw[color=dred, line width=0.5]
  (c2p)+(0,0)-- +(3*\w,0);
\draw[color=dred, line width=0.5]
  (c2r)+(0,0)-- +(0,-3*\h);
\draw[color=dred, line width=0.5]
  (c3p)+(0,0)-- +(3*\w,0);
  \draw[color=dred, line width=0.5]
  (c3r)+(0,0)-- +(0,-3*\h);

\end{tikzpicture}
}
\end{center}
\caption{\new{Safe} blocks are outlined in black\new{,} and red lines are the cores of their protective regions. Thicker red lines are the blocking structure. Even if all sites above the blocking structure are occupied, they cannot influence the configuration below the structure without help from \new{additional occupied sites below}.} 
\label{fig:blocking}
\end{figure}
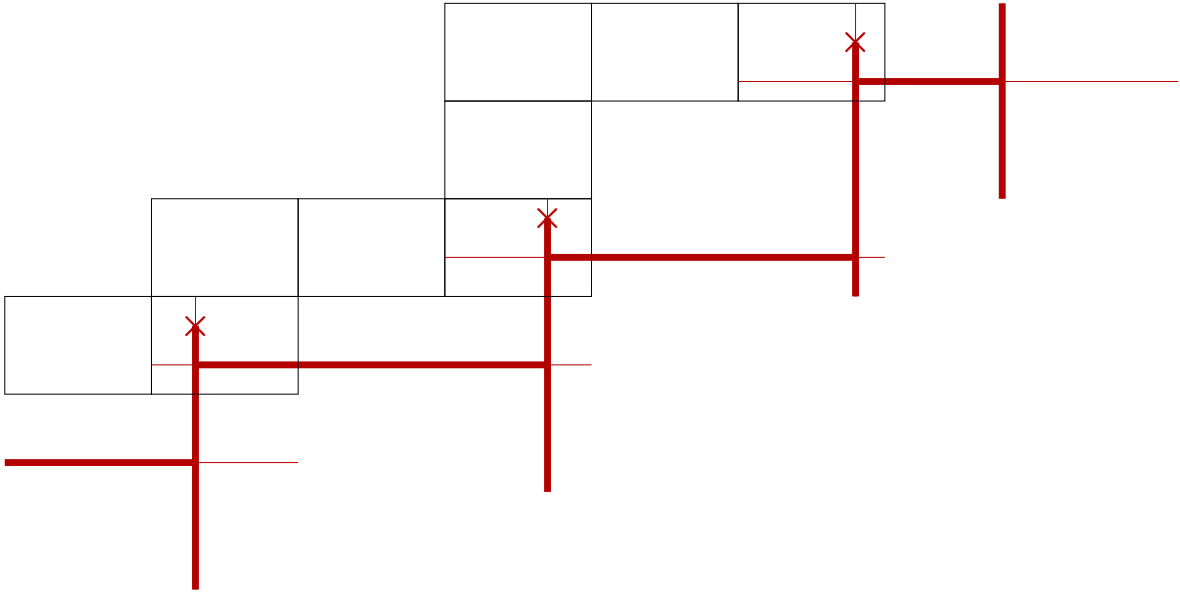

Assume that some fixed configuration $\xi_0$ has the following 
properties. 
There is a bi-infinite path $\ldots, z_{-1}, z_0, z_1, z_2,\ldots$ in $\bZ^2$ such that:
\begin{itemize}
    \item $R_{z_\ell}$ is a \new{safe block} for every $z_\ell$ in the path;
    \item %moving from left to right, 
    \new{the} path only makes horizontal or vertical steps in
positive directions: $z_{i+1}-z_i\in \{e_1,e_2\}$; and
    \item the path makes at most two consecutive steps in the same direction.
\end{itemize} 
We form a {\em blocking structure} from this path of \new{safe blocks} as follows. From the last box in each horizontal segment of the path, we include in the blocking structure the part of its core that lies (weakly) below and to the right of the closed site satisfying condition (1) of the \new{safe block}; see Figure~\ref{fig:blocking}. Let $A\subset\bZ^2$ be the region strictly above the blocking structure. Assume also $k\ge 3$, $m\ge 4$. \new{The following lemma encapsulates the key observation used to prove Theorem~\ref{thm:ub}. Its proof is similar to that of Proposition 3.2 in~\cite{GS2}, which is itself a simplified 2-dimensional version of Proposition 4 in~\cite{GHS}.}

%\note{strictly speaking, the blocking structure as defined here would include the thin red lines below and to the right of the thick red lines in Fig 3 - DJS}

\begin{lemma} \label{lemma:blocking} 
Under the above assumptions, convert the 
configuration in $A$ to entirely occupied and
to entirely closed, \new{respectively}, and run the two dynamics, obtaining the final configurations $\xi_o$ and $\xi_c$, \new{respectively}. Provided all connected clusters of 
occupied sites 
in $\xi_c$ have \new{$\ell^\infty$-diameter} at most $m/4$, \new{the configurations} $\xi_o$ and 
$\xi_c$ agree on $A^c$. 
\end{lemma}

\begin{proof}[Proof of Theorem 1.1]
Observe that $R_{z_1}$ and $R_{z_2}$ are \new{safe} independently if $z_1$
and $z_2$ are at $\ell^\infty$-distance at least $\max(k,m)$. Therefore, by \cite{LSS} and Lemma~\ref{lemma:protection}, the random configuration $\{z: R_z\text{ is \new{safe}}\}$
dominates an independent configuration on $\bZ^2$ with density arbitrarily close to $1$
if $C$ is large enough and $p$ is small. The rest of the proof is handled
by duality arguments developed in \cite{GH1, GH2, DDGHS}. In fact, a minor modification of the argument for Theorem 3.1 in \cite{GS2} (using Lemma~\ref{lemma:blocking} in place of Proposition 3.2 in \cite{GS2}) ends the proof. 
%\new{One required modification is that the definition of a safe block needs to be appropriately reflected for 
%protection in the other three directions.}
\end{proof}

We remark that we in fact need only vertical chimneys to construct a 
blocking structure and to combine them into a bounded protective 
set. Therefore, Theorem~\ref{thm:ub} remains valid if the  
dynamics also permits occupation of an open site with north and south 
occupied neighbor. Clearly, by monotonicity, Theorem~\ref{thm:lb} also remains valid in this case.

\section{Infecting the origin}

\new{ The key to our 
proof of Theorem~\ref{thm:lb} is the fact that, 
once a closed site is surrounded by occupied sites, 
it is effectively eliminated; a more precise statement is given in
the following lemma.

\begin{lemma}\label{lemma:key} Fix $x\in\bZ^2$ and two initial configurations on $\bZ^2$. Assume that the two initial configurations
agree off $x$ and both lead to eventual occupation of all four neighbors of $x$. Then the resulting final configurations also agree off $x$.
\end{lemma}

\begin{proof}
    The neighbors of $x$ become occupied eventually, so we may assume that they are initially occupied in both dynamics. Since no other vertex has $x$ as a neighbor, its state is now irrelevant for either of the dynamics.
\end{proof}
}

To infect the origin, we partition the lattice into square boxes of side length \new{$N =\lfloor n\cdot\frac{1}{p}\log(1/p)\rfloor$}: for $z\in\bZ^2$, $B_z=Nz+[0,N-1]^2$. Here, $n=n(p)$ is a quantity that goes to $\infty$ as $p\to 0$ and will be chosen later. 
We now call a box $S$ {\it good\/} if the following conditions hold.
\new{
\begin{itemize}
    \item[(G1)] No two closed sites in $S$ are within $\ell^\infty$-distance \new{$2\lceil n/p\rceil$}.
    \item[(G2)] Every closed site $x\in S$ has an occupied site in each of 
    the four lattice directions within distance $\lceil n/p\rceil$ \new{and inside $S$}.
    \item[(G3)] Every horizontal and vertical interval of length $\lfloor 3p^{-1}\log p^{-1}\rfloor$, included in $S$, contains an occupied site.
     \item[(G4)] 
    Every strip of width $\lfloor n p^{-1}\log p^{-1}\rfloor$ and 
    height $\lfloor n^2 p^{-1}\rfloor$, included in $S$,
    contains at most $n/4$ closed sites. 
    \item[(G5)] There are no closed sites within distance $2\lceil n/p\rceil$ from the 
    boundary of $S$.
      \item[(G6)] Every row and every column \new{within} $S$ contains at most 
      one closed site.
\end{itemize}
}
In fact, \new{(G6)} is not needed for the next lemma, but it is easy to achieve and we add it to make the proof a little simpler.

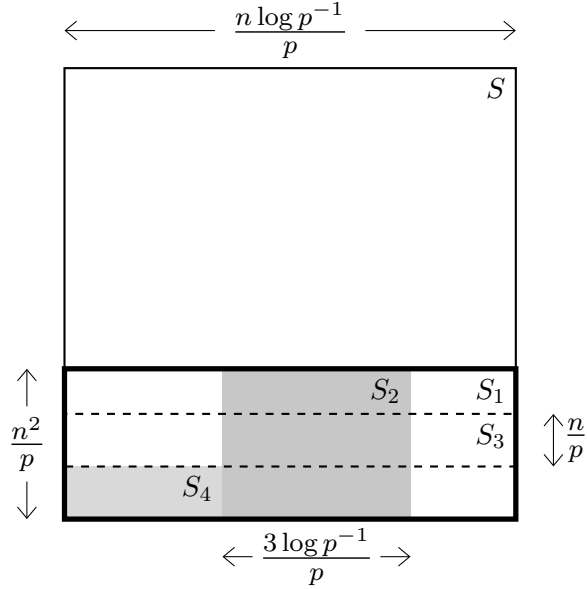
\begin{figure}%[ht!]
\begin{center}
\scalebox{1}{
\begin{tikzpicture}
%\draw[step=0.25,faint,very thin] (0,0) grid (2,4);
%\draw[step=1cm,black,thick] (0,0) grid (8,8);
%\fill[lgrey] (0.75,0) rectangle (1,1);
\newcommand{\st}{Straight Barb[length=1.5mm, width=2.5mm]}

\draw[color=black, thick] (0,0) rectangle (6,6);
\node [below left] at (6,6) {$S$};
\draw[color=black, thick] (0,0) rectangle (6,2);
\node [below left] at (6,2) {$S_1$};
\fill[mgrey] (2.1,0) rectangle (4.6,2);
\node [below left] at (4.6,2) {$S_2$};
%\draw [thick, dashed] (0,0.7) rectangle (6,1.4);
\node [below left] at (6,1.4) {$S_3$};
\fill[lgrey] (0,0) rectangle (2.1,0.7);
\node [below left] at (2.1, 0.7) {$S_4$};
\draw [ -\st] (4,6.5) -- (6,6.5);
\draw [ \st-] (0,6.5) -- (2,6.5);
\node at (3,6.5) {\scalebox{1.25}{$\frac{n\log p^{-1}}{p}$}};
% \draw [thick, ->] (-0.5,3.5) -- (-0.5,6);
% \draw [thick, <-] (-0.5,0) -- (-0.5,2.5);
% \node at (-0.65,3) {$\frac{n\log p^{-1}}{p}$};
\draw [ -\st] (-0.5,1.5) -- (-0.5,2);
\draw [ \st-] (-0.5,0) -- (-0.5,0.5);
\node at (-0.5,1) {\scalebox{1.25}{$\frac{n^2}{p}$}};

\draw[color=black, line width=0.07cm] (0,0) rectangle (6,2);
\draw [thick, dashed] (0,0.7) rectangle (6,1.4);

\draw [ -\st] (2.1+2.1,-0.5) -- (4.6,-0.5);
\draw [ \st-] (2.1,-0.5) -- (4.6-2.1,-0.5);
\node at (2.1+1.25,-0.5) {\scalebox{1.25}{$\frac{3\log p^{-1}}{p}$}};

\draw [ \st-\st] (6.5,0.7) -- (6.5,1.4);
%\draw [thick, <-] (6.5,0) -- (6.5,0.5);
\node at (6.8,1.05) {\scalebox{1.25}{$\frac{n}{p}$}};

\end{tikzpicture}
}
\end{center}
\caption{The box $S$ and its subregions. The name of each \new{subregion} is at its top \new{right} corner. \new{The outline of $S_1$ is thicker. Integer parts are again omitted.}} 
\label{fig:subregions}
\end{figure}

\new{The sites outside a box $S$ that have a neighbor in $S$ are 
divided into four sides, which we call the \em{outside boundary intervals} of $S$.}

\begin{lemma}\label{lemma:spread}
Under conditions \new{(G1)--(G6)}, if \new{an outside boundary interval of $S$} gets completely 
occupied, \new{then} every non-closed site of $S$ also gets occupied. 
\end{lemma}

\begin{proof} If there are no closed sites in $S$, the claim follows 
from the fact that every row and column of $S$ contains an occupied site by condition \new{(G3)}. 
Otherwise, we claim that there exists a closed site \new{in $S$}
such that all four of its neighbors eventually become occupied. 
\new{By Lemma~\ref{lemma:key}, we may eliminate this closed site by 
changing its status to occupied.}
As conditions \new{(G1)--(G6)} are preserved 
when a closed site is converted to occupied, the claim ends the proof, and in the remainder of the argument we establish this claim.

We may assume that the \new{outside interval} adjacent to the bottom of $S$ is occupied.
We denote $r=\lceil n/p\rceil$. \new{We will now define several subregions of $S$, illustrated in Figure~\ref{fig:subregions}.}
Let $S_1$ be the 
\new{$\lfloor n p^{-1}\log p^{-1}\rfloor\times \lfloor n^2 p^{-1}\rfloor$} strip on the bottom 
of $S$. We may assume that 
$S_1$ contains at least one closed site or it becomes completely occupied and we may replace $S_1$ with a vertical translate of $S_1$ by \new{$\lfloor n^2p^{-1}\rfloor$}. 
%\note{I changed the preceding from ``replace $S$ by $S\setminus S_1$," since the latter is no longer a square. -DJS} 
We conclude from \new{(G4)} that, within $S_1$, there is a strip $S_2$ of width at least 
$$
\frac{\new{\lfloor np^{-1}\log p^{-1}\rfloor}}{n/4+1}\ge 3p^{-1}\log p^{-1}
$$ 
and the same height as $S_1$ with no closed sites. By \new{(G3)}, $S_2$ gets 
entirely occupied. 
Also, as $$\frac{\new{\lfloor n^2p^{-1}\rfloor}}{n/4+1}>r,$$
there is a horizontal strip $S_3$ within $S_1$ of the same width as $S_1$ and height at least
$r$ that contains no closed sites. Let $S_4$ and $S_5$ be the left and right components of $S_1\setminus (S_2\cup S_3)$ below $S_3$. 
If \new{neither} of $S_4$, $S_5$ contains a closed site (in particular, if they are 
empty), then all the sites in $S$ 
up to the top of $S_3$ become occupied and 
we may again translate $S_1$ upwards and repeat the setup. Therefore, we assume, without loss of 
generality, that there is a closed site in $S_4$. (See Figure~\ref{fig:subregions}.)
\new{For any site $z\in S_4$ such that $z-e_2$ is also in $S_4$, we denote by $R(z)$ the rectangle with sides parallel to the axes, and with two of its corners at $z-e_2$ and the lower right corner of $S_4$. If $z-e_2$ is not in $S_4$, then we let $R(z)$ be the empty set.}

\begin{figure}[ht]
\begin{center}
\scalebox{0.8}{
\begin{tikzpicture}

\node at (-0.2,0.2) {\Huge\color{red}$\times$}; %{$x_\ell$};
\fill[lgrey] (-2.4,-5) rectangle (3,0);
\fill[lgrey] (0,0) rectangle (3,2.4);
\fill[dgrey] (-2.4,-6) rectangle (4,-5);
\fill[dgrey] (3,-6) rectangle (4,2.4);
\draw[color=black, ultra thick](-2.4,0) rectangle (2,0.4);
\draw[color=black, ultra thick](-0.4,-2) rectangle (0,2.4);

\fill [black] (-1.5, 0.2) circle (4pt);
\fill [black] (1.8, 0.2) circle (4pt);
\fill [black] (-0.2, -1.7) circle (4pt);
\fill [black] (-0.2, 1.2) circle (4pt);

\draw[color=black, ultra thick, dashed](-2.4,-5) rectangle (3,0);
 
\end{tikzpicture}
}
\end{center}
\caption{Elimination of a closed site, marked by the red $\times$, at $x$. The dark grey region becomes occupied either by the assumption (the bottom portion) or because it is a part of $S_2$
(the right portion). The light grey region --- the dashed 
rectangle is $R(x-re_1)$ --- includes no closed site, and thus also 
becomes occupied. Each of the four arms of the cross has length $r$, includes an initially occupied site (a black circle) but no closed sites, and thus also becomes occupied.} 
\label{fig:elimination}
\end{figure}
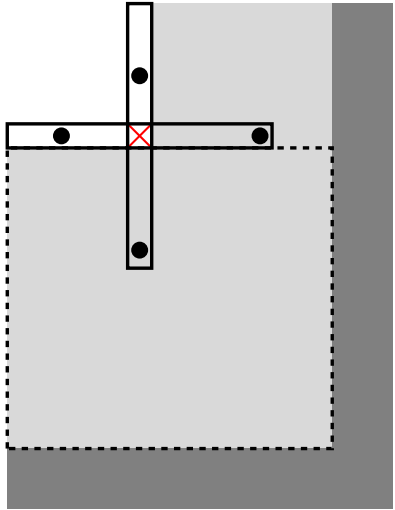

Pick the closed site  
$x_1\in S_4$ closest to 
$S_2$; note that $x_1$ is unique by \new{(G6)}. 
If $R(x_1-re_1)$ contains no closed site, then all four neighbors of $x_1$
become occupied by \new{(G1)}, \new{(G2) and (G5)}, and since the height of $S_3$ is at least $r$ (see Figure~\ref{fig:elimination}). Otherwise, choose the  unique closed site $x_2\in R(x_1-re_1)$ that is closest to $S_2$. By \new{(G1)}, $x_2$ is at distance more than $r$ from the horizontal line containing $x_1$. Then, if $R(x_2-re_1)$ contains no closed sites, all four neighbors of $x_2$
become occupied by \new{(G1)} and \new{(G2)}. Otherwise, we pick a closed site 
$x_3\in R(x_2-re_1)$, and proceed in the same fashion. As this process 
cannot continue indefinitely, we eventually find 
a closed site $x_\ell$, whose $R(x_\ell-re_1)$ contains no closed sites, and consequently the four neighbors of $x_\ell$ get occupied, as claimed.
%See Figure~\ref{fig:elimination}.
\end{proof}

\begin{lemma}\label{lemma:good-boxes}
    If \new{$n=\lfloor\log\log p^{-1}\rfloor$} and 
    $$q\le\frac{p^2}{\log p^{-1}\cdot n^4},$$
    then the probability that a box is good goes to $1$ as $p\to 0$.
\end{lemma}

\begin{proof}
We have, for some constant $C$, 
\begin{equation*}
\prob{\text{\new{(G1)} fails}}\le C\left(\left(\frac np\right)^2q\right)^2
\cdot \left(\frac{n\log p^{-1}}{p}\cdot \frac pn\right)^2
=C\left(\frac{n^2\log p^{-1}}{p^2}\cdot q\right)^2\le\frac{C}{n^{4}}\to 0.
\end{equation*}
Next, 
\begin{equation*}
\prob{\text{\new{(G2)} fails}}\le 
\left(\frac{n\log p^{-1}}{p}\right)^2\cdot q\cdot 4(1-p)^{n/p}
\le 4\,\frac{\log p^{-1}}{n^2}e^{-n}\le 
4\new{e}\,\frac{1}{n^2}
\to 0.
\end{equation*}
Next, again for some constant $C$,
\begin{equation*}
\begin{aligned}
\prob{\text{\new{(G3)} fails}}&\le C
\left(\frac{n\log p^{-1}}{p}\cdot \frac{n\log p^{-1}}{p}\cdot \frac{p}{\log p^{-1}}\right)\cdot (1-p)^{3\log p^{-1}/p}\\
&\le C\,\frac{n^2 \log p^{-1}}{p}e^{-3\log p^{-1}}=
Cn^2p^2\log p^{-1}
\to 0.
\end{aligned}
\end{equation*}
Next, for some constant $C$,
\begin{equation*}
\begin{aligned}
\prob{\text{\new{(G4)} fails}}&\le C
\left(\frac{n\log p^{-1}}{p}\cdot   \frac{p}{n^2}\right)\cdot  \left( 
 \frac{n^3\log p^{-1}}{p^2} \cdot q \right)^{n/4}\\
&\le C\log p^{-1}\cdot \left(\frac{1}{n}\right)^{n/4} \\
&\le C\left(\frac{e^{4}}{n}\right)^{n/4}
\to 0.
\end{aligned}
\end{equation*}
Finally, 
\begin{equation*}
\begin{aligned}
\prob{\text{\new{(G5)} fails}}\le
\new{C}\,
\left(\frac{n}{p}\right)^2\log p^{-1}\cdot  q\le Cn^{-2}
\to 0.
\end{aligned}
\end{equation*}
and 
\begin{equation*}
\begin{aligned}
\prob{\text{\new{(G6)} fails}}&\le C\,
\frac{n\log p^{-1}}{p}\cdot  \left(\frac{n\log p^{-1}}{p}\cdot q\right)^2 
\le \log p^{-1}\cdot p 
\to 0.
\end{aligned}
\end{equation*}
\end{proof}

\begin{proof}[Proof of Theorem~\ref{thm:lb}]
Note that the boxes $B_z$ are good independently for different $z$. 
Therefore, by Lemma~\ref{lemma:good-boxes}, with probability converging to $1$ as $p\to 0$,  the set $\{z\in\bZ^2:B_z\text{ is good}\}$ contains an infinite connected (in the usual $\bZ^2$ sense) set $\cC$
that includes the origin. Also with probability converging to $1$ as $p\to 0$,  $(0,0)\in R_0$ is not closed. With probability 1, there 
is a site $z\in\cC$ such that $R_z$ is initially completely occupied.
By Lemma~\ref{lemma:spread}, the origin eventually becomes occupied.
\end{proof}

\section*{Acknowledgments} \new{
The starting point of this work was a discussion between SL and Marek Biskup on percolation models that relate to modified bootstrap percolation. We thank Marek for many constructive suggestions.
JG was partially
supported by the Slovenian Research Agency research program
P1-0285 and  Simons Foundation Award \#709425. AEH is supported in part by a Royal Society Wolfson Fellowship. DS was partially supported by the NSF TRIPODS grant CCF--1740761 and the Simons Foundation.}

\end{document}